\documentclass[11pt]{amsart}
\usepackage{amssymb,amsmath}
\usepackage{graphicx}
\usepackage{amsmath}
\usepackage{amsfonts}
\usepackage{amsthm}
\usepackage{amssymb,latexsym}
\usepackage{fixmath}
\usepackage{mathrsfs,amsbsy}
\usepackage{color}
\usepackage{algorithm,algorithmic}
\usepackage{multirow}
\usepackage{url}
\usepackage{float}
\usepackage{subfigure}
\usepackage{bm}
\usepackage{caption}
\usepackage{makecell}
\usepackage{xr}
\usepackage{mathrsfs,amsbsy}
\usepackage{dsfont}
\usepackage{enumerate}
\usepackage{eucal}

\newtheorem{theorem}{Theorem}[section]
\newtheorem{lemma}[theorem]{Lemma}
\date{\today}
\title[]{ Construction of Discontinuous Enrichment Functions for Enriched FEM's for Interface Elliptic Problems in 1D}
\author{}
\date{\today}							
\author{So-Hsiang Chou$^\dagger$ and C. Attanayake$^\ddagger$}
\thanks{$\dagger$
 Department of Mathematics and Statistics, Bowling Green
State University, Bowling Green, OH, 43403-0221. {{\tt email:chou@bgsu.edu}};\,$\ddagger$ Department of Mathematics,
Miami University
Middletown, OH 45042. {{\tt
e-mail:attanac@muohio.edu}} }

\begin{document}
\maketitle
\begin{abstract}
We introduce an enriched unfitted finite element method to solve 1D elliptic interface problems with discontinuous solutions, including those having implicit or Robin-type
interface jump conditions. We present a novel approach to construct a one-parameter family of discontinuous enrichment functions by finding an optimal order interpolating function to the discontinuous solutions. In the literature, an enrichment function is usually given beforehand, not related to the construction step of an interpolation operator. Furthermore, we recover the well-known continuous enrichment function when the parameter is set to zero. To prove its efficiency, the enriched linear and quadratic
elements are applied to a multi-layer wall model for drug-eluting stents in which zero-flux jump conditions and implicit concentration interface conditions are both present.
\end{abstract}
{\bf \small Key Words.}\keywords{\small{ enriched finite element, elliptic interface, implicit interface jump condition, Robin interface jump condition, linear and quadratic finite elements.}
\maketitle
\section{Introduction}
Consider the interface two-point boundary value problem
\begin{equation}
\begin{cases}
-(\beta(x)p'(x))'+w(x)p(x)= f(x), &  x\in I=(a,b),  \label{eqn:First}\\
p(a)= p(b) = 0,
\end{cases}
\end{equation}
where $w(x)\geq 0$, and  $0<\beta\in C[a,\alpha]\cup C[\alpha,b]$ is discontinuous across the interface $\alpha$
with the jump conditions on $p$ and its flux $q:=\beta p'$:
\begin{align}\label{implicit}
[p]_{\alpha} &=\lambda F(q^+,q^-,[p']_\alpha),\quad \lambda\in \mathbb R,\, F:[c,d]\to \mathbb R, \\
\label{implicit2}
[\beta p']_\alpha &= g,  \quad g\in\mathbb R
\end{align}
where the jump quantity
 \[    [s]_\alpha:=s(\alpha^+)-s(\alpha^-), \quad s^{\pm}:=s(\alpha^\pm):=\lim_{\epsilon\to 0^+} s(\alpha\pm \epsilon).\]
The primary variable $p$ may stand for the pressure, temperature, or concentration in a medium with certain physical properties
and the derived quantity $q:=-\beta p'$ is the corresponding Darcy velocity, heat flux, or concentration flux, which is equally important.
 The piecewise continuous $\beta$ reflects a nonuniform  material or medium property {\it{(we do not require $\beta$ to be piecewise constant}}).  The function
$w(x)$ reflects the surroundings of the medium. The case of $\lambda=0$ is widely studied, while the case of $\lambda>0$ gives rise to a more difficult situation. For example, the case of rightward concentration flow \cite{Wang, Zhang1, Zhang2} imposes
\begin{equation}
\begin{cases}\label{jmp1}
[p]_\alpha&=\lambda (\beta p)'(\alpha^-)\\
[\beta p']_\alpha &= 0,
\end{cases}
\end{equation}
which generates an implicit condition since the left-sided derivative is unknown. Implicit interface conditions abound in higher dimensional applications \cite{Ammari, Hahn, Kim, Kruit1}.  For definiteness, we will study a class of efficient enriched methods for problem  \eqref{eqn:First} under the jump conditions \eqref{jmp1}, but our methods apply to problem \eqref{eqn:First} subject to the general conditions \eqref{implicit}-\eqref{implicit2} with a well-posed weak formulation. After a simple calculation, it is easy to
see that \eqref{jmp1} is equivalent to
\begin{equation}
\begin{cases}\label{jmpGood}
[p]_\alpha&=\gamma[p']_\alpha, \quad \gamma=-\frac{\lambda\beta^-\beta^+}{[\beta]_\alpha}\\
[\beta p']_\alpha &= 0,
\end{cases}
\end{equation}
which is indeed of the type \eqref{implicit}-\eqref{implicit2}.

Numerical methods for the interface problem (\ref{eqn:First}) under \eqref{jmp1} generally use meshes that are either fitted or unfitted with the interface. A method allowing unfitted meshes would be very efficient when one has to follow a moving interface \cite{he2013immersed} in a temporal problem.  For the unfitted methods, there are available geometrically unfitted
finite element methods typified in \cite{Bordas} and the reference therein, the immersed finite and finite difference methods \cite{Chou, guo2019group, guo2018nonconforming, jo2019recent, LI, LI:2006, Li2003}, the stable generalize finite element methods (SGFEM) \cite{babuvska2004generalized, Babuska1, Babuska2, Deng, zhang2021generalized}, among others. In an unfitted method, the mesh is made up of interface elements where the interface intersects elements and non-interface elements where the interface is absent. On a non-interface element, one uses standard local shape functions, whereas on an interface element one uses specialized local shape functions reflecting the jump conditions. For an enriched method, the standard finite elements are enriched
with some enrichment functions that reflect the presence of the interfaces. It was originally designed to handle crack problems \cite{Belyt, Fries, Moes}, but for recent years efforts have been made to generalize it to fluid problems, see \cite{Wang} and the references therein.

The construction of the local shape basis of an immersed finite element or finite difference method uses information on discontinuous $\beta$ while an enriched method does not. Thus an enriched method does not require the discontinuous diffusion coefficient to be piecewise
constant, which is an advantage. On the other hand, it makes the choice of the enrichment function less intuitive and the error analysis arguably harder. The purpose of this
paper is to propose an approach to constructing the enrichment function from optimal order error analysis. The general idea is as follows. In the error analysis, we use the principle that
says, roughly, the error in the finite element solution $p_h$ should be bounded by the approximation error in the finite element space $V_h$:
\begin{equation}\label{motive}
     ||p-p_h||\leq C \inf_{\chi\in V_h}||p-\chi||+\text{ consistency error}.
     \end{equation}
Suppose that the consistent error is of optimal order, then the optimal order analysis is completed if we can demonstrate an optimal order approximate piecewise polynomial from $V_h$. In an enriched finite element method, $V_h$ takes the form of
\[  V_{h}:=S_h\oplus \psi S_h=\{p_{h}+q_{h}\psi:\,p_{h},q_{h}\in S_{h}\},\]
where $S_h$ is a standard finite element space (e.g., $\mathbb P_k$-conforming, $k\geq 1$), and the function $\psi$ is an enrichment function to reflect the jump conditions. In the literature,
$\psi$ is usually given beforehand and then one tries to find the optimal order interpolating polynomial to prove convergence. Our new approach is to connect the construction of $\psi$ and the interpolating
polynomial together and finds $\psi$ through error analysis. In this way, we also have a unified theory for constructing enrichment functions for continuous and discontinuous finite element solutions.

Let's mention how we were motivated to come up with the new approach. In view of \eqref{motive}, for standard continuous conforming finite element methods there are familiar interpolating polynomials that do the job \cite{Ern1}. For problem \eqref{eqn:First} with a continuous solution ($[p]_\alpha=0$), Deng \cite{Deng} proved the convergence of finite element solutions in {\it all} $\mathbb P_i-$conforming spaces($i\geq 1$) enriched by the same well-known hat function \cite{babuvska2004generalized, Babuska1, Babuska2, Deng} (cf. Eq. \eqref{eqn:lin} below). The crux of the proof was again the existence of a simple interpolating polynomial.
However, for an enriched \emph{immersed or unfitted} method approximating a discontinuous solution, it is impossible to find the same type of interpolation operator due to the finite jump of $[p]_\alpha$ (cf. \cite{Choun}). On the other hand, we in \cite{ATTChou2021}, unaware of \cite{Deng}, used a different interpolation operator to prove optimal order convergence. The approach was motivated by an additional presence of an interface deviation. In this paper we generalize the analysis of \cite{ATTChou2021}; modifying that operator ($I^c_h$ of \eqref{defIhp} below) to find the desired interpolating polynomial. The new family of the enrichment functions is a result of this analysis, not given beforehand. However, the formula of the enrichment function (cf. \eqref{eqn:lin1} below) is simple and intuitive, and can be used without knowing the detail of the analysis.

The organization of this paper is as follows. In Section 2, we state the weak formulation for the implicit interface condition problem, define enrichment functions and spaces, and put their role in perspective in Remark 2.1. In Section 3, we carry out the error analysis and show how the construction of the enrichment function is related to it. Optimal order convergence in the broken $H^1$ and $L^2$-norms is given in Theorem \ref{thm3.2}. In addition, the second-order accuracy of $p_h$ at the nodes is proven
  in Theorem \ref{2order}. In Section 4, we provide numerical examples of a porous wall model to demonstrate the effectiveness
of the present enriched finite element method and confirm the convergence theory. Furthermore, following the viewpoint of the SGFEM \cite{Babuska1, Babuska2, Deng}, we compare the condition numbers of our (discontinuous solution) method with those in the continuous solution case \cite{ATTChou2021}, and numerically show that they are comparable for the same mesh sizes. Both linear and quadratic enriched elements are tested. Finally, in Section 5 we give some concluding remarks and discuss possible extensions of the present approach to multiple dimensions.
\section{Enrichment Functions and Spaces}
\subsection{Weak Formulation}
Let $I^-=(a,\alpha)$ and $I^+=(\alpha,b)$, and define
\[  H_{\alpha,0}^1(I)=\{v\in L^2(I): v\in H^1(I^-)\cap H^1(I^+), v(a)=v(b)=0\}.\]
We use conventional Sobolev norm notation. For example,  $|u|_{1,J}$  denotes the usual $H^1$-seminorm for $u\in H^1(J)$, and $||u||^2_{i,I^-\cup I^+}=||u||^2_{i,I^-}+||u||^2_{i, I^+}, i=1,2$ for $u\in  H_\alpha^2(I)$, where
\[
 H_\alpha^2(I)= H^2(I^-)\cap H^2(I^+).\]
 The space $H_{\alpha,0}^1(I)$ is endowed with the $||\cdot||_{1,I^-\cup I^+}$ norm, and $H_\alpha^2(I)$ with the $||\cdot||_{2,I^-\cup I^+}$ norm.
With this in mind, the weak formulation of the problem \eqref{eqn:First} under (\ref{jmp1}) is: Given $f \in L^{2}(I)$, find $p\in H_{\alpha,0}^{1}(I)$ such that
\begin{equation}\label{ExactWeak}
a(p,q) = (f,q)\quad \forall q\in H_{\alpha,0}^{1}(I),
\end{equation}
where
\begin{align*}
a(p,q) =& \int_{a}^{b}\beta(x)p'(x)q'(x)dx+\int_{a}^{b}w(x)p(x)q(x)dx+\frac{[p]_\alpha[q]_\alpha}{\lambda},\\
(f,q) =& \int_{a}^{b}f(x)q(x)dx.
\end{align*}
 The above weak formulation can be easily derived by integration-by-parts and by \eqref{jmp1}. Since $\lambda>0$, the bilinear form $a(\cdot,\cdot)$ is coercive and is bounded due to Poincar$\acute{e}$ inequality. By the Lax Milgram theorem, a unique solution $p$ exists. Throughout the paper, we assume that the functions $\beta$, $f$, and $w$ are such that
 the solution $p\in H^2_\alpha(I)$.

 \subsection{Enrichment Functions}
 We now introduce an approximation space for the solution $p$. Let $a=x_{0}<x_{1}<\ldots <x_{k}<x_{k+1}<\ldots<x_{n}=b$ be a partition of $I$ and the interface point
$\alpha\in(x_{k},x_{k+1})$ for some $k$. As usual, the meshsize $h:=\max_i h_i, h_i=x_{i+1}-x_i,i=0,\ldots, n-1$. Define the enrichment function
\begin{equation}
\psi(x) :=
\begin{cases}
0   &  x\in [a,x_{k}]\\
m_1(x-x_k)&  x\in [x_{k},\alpha)  \label{eqn:lin1}\\
m_2(x-x_{k+1}) &  x\in (\alpha,x_{k+1}]  \\
0   &  x\in [x_{k+1},b]
\end{cases}
\end{equation}
where
\begin{equation}\label{slopes}
m_1=\dfrac{\alpha-x_{k+1}}{x_{k+1}-x_{k}},\quad   m_2=\dfrac{(\alpha-x_k-\gamma)(\alpha-x_{k+1})}{(x_{k+1}-x_k)(\alpha-x_{k+1}-\gamma)},\quad \gamma=-\frac{\lambda\beta^-\beta^+}{[\beta]_\alpha}
\end{equation}
{\bf Remark 2.1.}
\begin{itemize}
\item Note that  $\psi$ satisfies the following conditions:
\begin{equation}\label{psi}
  \psi(x_{k}) = \psi(x_{k+1})=0, \qquad
 \quad  [ \psi']_\alpha\neq 0,\qquad [\psi]_\alpha=\begin{cases}0\qquad \quad\, \quad\gamma=0\\
\text{ nonzero}\quad \gamma\neq 0.\end{cases}
\end{equation}
\item We can view $\psi$ as parameterized by $\gamma$, and for different problems we would have to define $\gamma$. In the present application case,
$\gamma=-\lambda \beta^-\beta^+/[\beta]_\alpha$. Our theory depends on $\gamma$, not its specific definition.
\item If we set $\gamma=0\, ([p]_\alpha=0)$, we recover the familiar continuous enriched function \cite{Babuska1, Babuska2, Deng, ATTChou2021} for the continuous case in which $[\psi']_\alpha=1$ (cf. Eq.\eqref{eqn:lin}).\\ In other words, the continuous case is the limiting case of the discontinuous ones.
\item Notice that the slopes are uniformly bounded, i.e., there exists a constant $C>0$ such that
\begin{equation}\label{uniform}
|m_|+|m_2|\leq C \quad \forall x_k,x_{k+1},\alpha, 0\leq h\leq 1.
\end{equation}
\item The main gist of this paper is to obtain $\psi$ as a natural consequence of our error analysis. The definition of $m_2$ is a result of zeroing out of infinite coefficient of $[p']_\alpha$ in the error analysis (cf. Eq. \eqref{J5}).
\end{itemize}

 Let us describe the enriched space associated with $\psi$. Let $\bar I=\cup_0^{n-1}I_{i},I_i=[x_i,x_{i+1}]$ and let $S_{h}$ be the conforming linear finite element space
\begin{align}\label{P1}
S_{h} &= \{
v_{h}\in C(\bar I) : v_{h}|_{I_i} \in \mathbb P_{1}, i=0,\ldots,n-1, v_h(a)=v_h(b)=0
\}\\
\notag
&=\text{span}\{\phi_i,i=1,2,\ldots,n-1\}
\end{align}
where $\phi_i$'s are the Lagrange nodal basis (hat) functions.
We denote the usual $\mathbb P_1$-interpolation operator by $\pi_h:C(\bar I)\to S_h$,
\[     \pi_h g=\sum_{i=1}^{n-1}g(x_i)\phi_i,\]
and define the enriched finite element space
\begin{align}\label{enriched}
\overline{S}_{h}&=S_h\oplus \psi S_h=\{p_{h}+q_{h}\psi:\,p_{h},q_{h}\in S_{h}\}\\
&=\text{span}\{\phi_1,\phi_2,\ldots,\phi_{n-1},\phi_k\psi,\phi_{k+1}\psi\}.\notag
\end{align}
Consider the enriched finite element method for problem (\ref{eqn:First}): Find $p_{h}\in \overline{S}_{h}\subset H^1_{\alpha,0}$  such that
\begin{equation}\label{FEMweak}
a(p_{h},q_{h}) = (f,q_{h})\quad \forall q_{h}\in \overline{S}_{h}.
\end{equation}

\subsection{Optimal Order Interpolating Polynomial $I_hp$}
It is essential for the enriched space to have good approximation properties for the functions in $H_\alpha^2(I)$ that satisfy the jump conditions (\ref{jmpGood}).
 For $p\in  H_\alpha^2(I)$, let $p_i,i=1,2$ be the extensions
of $p$ restricted to $I^-$ and $I^+$ to $H^2(I)$, respectively \cite{Hansbo}. Thus $p'_{2}-p'_{1}$ is in $H^1(I)\subset C(\bar I)$ due to the Sobolev inequality, and as a result
the usual $\mathbb P_1$--interpolation operator $\pi_h(p_2^\prime-p_1^\prime)\in S_h$ is well defined.
To exhibit approximation properties of $\bar S_h$ for functions in $H^2_\alpha(I)$ that satisfy \eqref{jmpGood}, we first define the interpolation operator $I^c_h: H^2_\alpha(I)\to \overline{S}_{h}$
\begin{align}\label{defIhp}
I^c_{h}p =& \pi_h p+\pi_h(p'_{2}-p'_{1})\psi.
\end{align}
In particular
\begin{equation}\label{formula}
I^c_{h}p =\dfrac{p_1(x_k)(x_{k+1}-x)+p_2(x_{k+1})(x-x_k)}{x_{k+1}-x_k}+\pi_h(p'_{2}-p'_{1})\psi(x)\quad \forall x\in [x_k,x_{k+1}].
\end{equation}
 This interpolation operator has been used successfully in \cite{ATTChou2021}, but our experience showed that it is not capable of handling
the discontinuous case $[p]_\alpha\neq 0$. To emphasize we use a superscript $c$ to indicate continuity.
Now we modify it with an added correction term to accommodate the case of $[p]_\alpha\neq 0$:
Define the interpolation operator $I_h: H_\alpha^2(I)\to \overline{S}_{h}$
\begin{align}\label{defIhd}
I_{h}p =& I^c_{h}p+\delta\psi,
\end{align}
where
\begin{align}\label{delta}
\delta&=-\frac{h_k^{-1}[p]_\alpha}{m_1}=-\frac{[p]_\alpha}{\alpha-x_{k+1}}.
\end{align}

 The $\delta$-term is motivated by the error analysis in Lemma \ref{lemma3trm} below. Its presence is to kill the jump term in $p$ across $\alpha$ that may go to infinity as $h$ goes to zero (See Eq. \eqref{cancel}).

Let $\chi_i,i=1,2$ be the characteristic functions of $I^-$ and $I^+$, respectively, and let
\begin{align*}
V_h&:=\{v=v_{h,1}\chi_1+v_{h,2}\chi_2; \,\,v_{h,i}\in S_h,i=1,2\}.
\end{align*}
Note that functions in the above space may be discontinuous at $\alpha$.
Define the auxiliary  interpolations $\bar{I}_{h}: H_\alpha^2(I)\to V_h$,
\begin{align*}
\bar{I}_{h}p&=\pi_h p_1\chi_1+\pi_h p_2\chi_2.
\end{align*}
To derive a bound for the term $|p-I_{h}p|_{1,I^{-}\cup I^{+}}$ we split the error as follows:
\begin{equation} \label{MainError}
|p-I_{h}p|_{1,I^{-}\cup I^{+}} \leq |p-\bar{I}_{h}p|_{1,I^{-}\cup I^{+}} + |\bar{I}_{h}p-I_{h}p|_{1,I^{-}\cup I^{+}}.
\end{equation}
From the classical approximation theory
\begin{equation}\label{ClasicError}
|p-\bar{I}_{h}p|_{1,I^{-}\cup I^{+}}  \leq Ch\|p\|_{2,I^{-}\cup I^{+}}.
\end{equation}
Thus it suffices to estimate the second term on the right side of (\ref{MainError}), which is done in the following two lemmas.
We mention in passing that all the constants in the estimates should be independent of the interface position as well.
This fact is important if one wants to use the method for moving interface problems.

\section{Construction of Enrichment Functions in relation to Error Analysis}
\begin{lemma}\label{lemma3trm}
There exists a  positive constant $C$ independent of $h$ and $\alpha$ such that
\begin{equation}\label{estimate1}
|\bar{I}_{h}p-I_{h}p|_{1,I^{-}}  \leq Ch\|p\|_{2,I^{-}\cup I^{+}})\quad \forall p\in {H}_\alpha^{2}(I).
\end{equation}

\end{lemma}

\begin{proof} It suffices to show the detailed analysis on the interface element $[x_k,x_{k+1}]$.
For the interval $[x_{k},\alpha]$, from the definition (\ref{defIhp}) of $I^c_{h}p$ and the addition and substraction of the same quantity yield
\begin{align}\label{deferrLem2}
&\left(\bar{I}_{h}p - I^c_{h}p\right)' = \left(\bar{I}_{h}p -\pi_{h}p \right)'  - \left( \pi_h(p'_{2}-p'_{1})(x)\psi(x)\right)' \notag \\
&=J_1-J_2,
\end{align}
where
\begin{align}
J_1&:= \left(\bar{I}_{h}p -\pi_{h}p \right)'  +  \frac{(x_{k+1}-\alpha)(p'_{2}(\alpha) - p'_{1}(\alpha))}{x_{k+1}-x_{k}}\notag \\
J_2&:= \frac{(x_{k+1}-\alpha)(p'_{2}(\alpha) - p'_{1}(\alpha))}{x_{k+1}-x_{k}}  +  \left( \pi_h(p'_{2}-p'_{1})(x)\psi(x)\right)'.\notag
\end{align}

Since $\bar{I}_{h}p=\pi_{h}p_{1}\chi_1$ on $[x_{k},\alpha]$, the first term in $J_1$
\begin{align*}
\left(\bar{I}_{h}p -\pi_{h}p \right)'(x)
& =  \frac{p_{1}(x_{k+1}) -  p_{2}(x_{k+1})}{x_{k+1}-x_{k}}\\
 &= \frac{p_{1}(x_{k+1}) - p_{1}(\alpha) + p_{2}(\alpha) -  p_{2}(x_{k+1})}{x_{k+1}-x_{k}}-\frac{[p]_\alpha}{x_{k+1}-x_{k}},
\end{align*}
and combining this with the second term in $J_1$ leads to
\begin{equation}\label{J3}
J_1=J_3-\frac{[p]_\alpha}{x_{k+1}-x_{k}},
\end{equation}
where using the Taylor's expansion with integral remainder form
\begin{align} \label{term1Lem2}
|J_3|&=\bigg|
\frac{p_{1}(x_{k+1}) - p_{1}(\alpha) + p_{2}(\alpha) -  p_{2}(x_{k+1})}{x_{k+1}-x_{k}} + \frac{(x_{k+1}-\alpha)(p'_{2}(\alpha) - p'_{1}(\alpha))}{x_{k+1}-x_{k}}
\bigg|
\notag \\
&=\bigg|
\frac{p_{1}(x_{k+1}) - p_{1}(\alpha) - p'_{1}(\alpha)(x_{k+1}-\alpha)}{x_{k+1}-x_{k}} -\frac{ p_{2}(x_{k+1})-p_2(\alpha)-p'_2(\alpha)((x_{k+1}-\alpha)}{x_{k+1}-x_{k}}
\bigg|\\
&=\frac{1}{x_{k+1}-x_{k}}\bigg|
\left(
\int_{\alpha}^{x_{k+1}}p_{1}''(t)(x_{k+1}-t)dt - \int_{\alpha}^{x_{k+1}}p_{2}''(t)(x_{k+1}-t)dt
\right)
\bigg|
\notag \\
&\leq
\frac{x_{k+1}-\alpha}{x_{k+1}-x_{k}}\left(
\int_{\alpha}^{x_{k+1}}|p_{1}''(t)|dt +\int_{\alpha}^{x_{k+1}}|p_{2}''(t)|dt
\right)
\notag \\
& \leq 2\frac{x_{k+1}-\alpha}{x_{k+1} - x_{k}} h^{1/2} \|p''\|_{0,I^{-}\cup I^{+}} \notag \\
\notag&\leq Ch^{1/2} \|p''\|_{0,I^{-}\cup I^{+}},
\end{align}
where the constant $C=2$, independent of $h$ and $\alpha$.
Note that $J_1$ is the difference between a small quantity $J_3$ and a large quantity $[p]_\alpha/(x_{k+1}-x_{k})$ as $h$ goes to zero.
The latter is controlled by the $\psi'$ terms in \eqref{delta} through the $\delta-$parameter in the following relation
\begin{align}\label{cancel}
   \left(\bar{I}_{h}p - I_{h}p\right)'&=\left(\bar{I}_{h}p - I^c_{h}p\right)'-\delta\psi'\\\notag
   &=J_1-J_2-\delta\psi'\\\notag
   &=J_3-J_2-\frac{[p]_\alpha}{x_{k+1}-x_k}-\delta m_1\\
\notag   &=J_3-J_2
\end{align}
by the way we defined $\delta$ in \eqref{delta}. Next we show that $J_2$ is the difference between a small quantity and a large term we can control.

To avoid clustering of expressions, let $\Delta:=p_2-p_1$ so that $\Delta'=p_2'-p_1'$ and $\Delta''=p_2''-p_1''$. We also
denote $\Delta'(x_k)$ by $\Delta'_k$, and  $ \Delta'(x_{k+1})$ by $\Delta'_{k+1}$. Below, we use these notations when necessary.
 First note that with $\psi=m_1(x-x_k)$ and $\Delta'(\alpha)=[p']_\alpha$
\begin{align}\notag
&\left( \pi_h(p'_{2}- p'_{1})(x)\psi\right)' \\
\notag
&=h_k^{-1}
(\Delta'_{k+1}- \Delta'_k)\psi
+h_k^{-1}\big[(\Delta'_{k+1}(x-x_k)-
\Delta'_{k}(x-x_{k+1})
\big]\psi'
\\& =m_1h_k^{-1}\big[
2\Delta'_{k+1}(x-x_{k}) - \Delta'_{k}(x-x_{k}) - \Delta'_{k}(x-x_{k+1})
\big]\notag
\end{align}
and hence
\begin{align*}
J_2&:= \frac{(x_{k+1}-\alpha)(p_{2}'(\alpha) - p_{1}'(\alpha))}{x_{k+1}-x_{k}}+\left( \pi_h(\Delta')(x)\psi(x)\right)'\\
&=h_k^{-1}(x_{k+1}-\alpha)[p']_\alpha\\
&\quad +m_1h_k^{-1}\Big[
2\Delta'_{k+1}(x-x_{k}) - \Delta'_{k}(x-x_{k}) - \Delta'_{k}(x-x_{k+1})\Big]\\
&=h_k^{-1}(x_{k}-\alpha)[p']_\alpha+[p']_\alpha\\
&\quad +m_1h_k^{-1}\Big[
2\Delta'_{k+1}(x-x_{k}) - \Delta'_{k}(x-x_{k}) - \Delta'_{k}(x-x_{k+1})\Big]\\
&=h_k^{-1}(x_{k}-\alpha)[p']_\alpha+[p']_\alpha+
m_1h_k^{-1}\Big[\sum_{i=1}^4I_i+[p']_\alpha h_k\Big],\\
&=h_k^{-1}(x_{k}-\alpha)[p']_\alpha+(m_1+1)[p']_\alpha+
m_1h_k^{-1}\sum_{i=1}^4I_i\\
&:=J_4=m_1h_k^{-1}\sum_{i=1}^4I_i  \quad (\text{ by } \eqref{slopes})
\end{align*}
where
\begin{align*}
I_1&=(\Delta'(x_{k+1})-\Delta'(x_{k}))(x-x_{k+1}),\\
I_2&=(\Delta'(x_{k+1})-\Delta'(x_{k}))(x-x_k),\\
I_3&=(\Delta'(x_k)-\Delta'(\alpha))(x_{k+1}-x_k),\\
I_4&=(\Delta'(x_{k+1})-\Delta'(x_{k}))(x_{k+1}-x_{k}).
\end{align*}
Each of the $m_1h_k^{-1}I_i$ terms in $J_4$ can be estimated similarly using the Cauchy-Schwarz inequality, e.g.,
\begin{align*}
  |m_1h_k^{-1}I_1|&\leq |m_1|\frac{x_{k+1}-x}{x_{k+1}-x_{k}}\int_{x_{k}}^{x_{k+1}} |(p_2'-p_1')'|(y)dy\\
 &\leq |m_1|\frac{x_{k+1}-x}{x_{k+1}-x_{k}}\left(\int_{x_{k}}^{x_{k+1}} |(p_2'-p_1')'|^2(y)dy\right)^{1/2}(x_{k+1}-x_k)^{1/2}\\
 &\leq |m_1|h^{1/2}||p''||_{0, I^-\cup I^+}\\
 &\leq Ch^{1/2}||p''||_{0, I^-\cup I^+}, \qquad   (|m_1|\leq 1)
 \end{align*}
 where $C$ is a constant independent of $h$ and $\alpha$.
Combining these estimates we see that
\begin{align}\label{J4}
J_2& =J_4
\end{align}
with
\begin{align}
|J_4|
= |m_1h_k^{-1}\sum_{i=1}^4I_i|\leq Ch^{1/2}||p''||_{0, I^-\cup I^+}.
\end{align}

From \eqref{cancel} and using \eqref{J3}, \eqref{J4}, and \eqref{delta}.
\begin{align}\label{cancel1}
   \left(\bar{I}_{h}p - I_{h}p\right)'&=\left(\bar{I}_{h}p - I^c_{h}p\right)'-\delta\psi'\\
  \notag         &=J_3-J_4.
      \end{align}
Gathering all the local estimates and integrating, we have
\begin{equation}\label{estimate1}
|\bar{I}_{h}p-I_{h}p|_{1,I^{-}}  \leq Ch\|p\|_{2,I^{-}\cup I^{+}}\quad  \forall p\in {H}_\alpha^{2}(I)
\end{equation}
where $C$ is independent of $h$ and $\alpha$.
\end{proof}

\begin{lemma}\label{lemma3right1}
There exist a positive constant $C$ independent of $h$ and $\alpha$ such that
\[
|\bar{I}_{h}p-I_{h}p|_{1,I^{+}}  \leq C h\|p\|_{2,I^{-}\cup I^{+}}\quad \forall p\in {H}_\alpha^{2}(I) \text{ satisfying } \eqref{jmpGood}.
\]
\end{lemma}

\begin{proof}
For the interval $[\alpha,x_{k+1}]$, from the definition (\ref{defIhp}) of $I^c_{h}p$ and adding and subtracting of the same quantity, $h_k^{-1}(x_{k+1}-\alpha)[p']_\alpha$, yield
\begin{align}\label{deferrLem2}
&\left(\bar{I}_{h}p - I^c_{h}p\right)' = \left(\bar{I}_{h}p -\pi_{h}p \right)'  - \left( \pi_h(p'_{2}-p'_{1})(x)\psi(x)\right)' \notag \\
&= \left(\left(\bar{I}_{h}p -\pi_{h}p \right)'  +  \frac{(x_{k}-\alpha)(p'_{2}(\alpha) - p'_{1}(\alpha))}{x_{k+1}-x_{k}} \right) \notag \\
&\quad  -\left( \frac{(x_{k}-\alpha)(p'_{2}(\alpha) - p'_{1}(\alpha))}{x_{k+1}-x_{k}} +  \left( \pi_h(p'_{2}-p'_{1})(x)\psi(x)\right)' \right)\notag\\
&:=\tilde J_1-\tilde J_2.
\end{align}

Noting that $\bar{I}_{h}p=\pi_{h}p_{2}\chi_2$ for $x\in [\alpha,x_{k+1}]$, we see that
\begin{align*}
\left(\bar{I}_{h}p -\pi_{h}p \right)'(x)
& =  \frac{p_{1}(x_{k}) -  p_{2}(x_{k})}{x_{k+1}-x_{k}}\\
 &= \frac{p_{1}(x_{k}) - p_{1}(\alpha) + p_{2}(\alpha) -  p_{2}(x_{k})}{x_{k+1}-x_{k}}-\frac{[p]_\alpha}{x_{k+1}-x_{k}},
\end{align*}
and hence \begin{equation}\label{J1T}
\tilde J_1=\tilde J_3-\frac{[p]_\alpha}{x_{k+1}-x_{k}}
\end{equation}
where
\begin{align}\notag
|\tilde J_3|&=\bigg|
\frac{p_{1}(x_{k}) - p_{1}(\alpha) + p_{2}(\alpha) -  p_{2}(x_{k})}{x_{k+1}-x_{k}} + \frac{(x_{k}-\alpha)(p'_{2}(\alpha) - p'_{1}(\alpha))}{x_{k+1}-x_{k}}
\bigg|
\notag \\
\notag&=\bigg|
\frac{p_{1}(x_{k}) - p_{1}(\alpha) - p'_{1}(\alpha)(x_{k}-\alpha)}{x_{k+1}-x_{k}} -\frac{ p_{2}(x_{k})-p_2(\alpha)-p'_2(\alpha)((x_{k}-\alpha)}{x_{k+1}-x_{k}}
\bigg|\\
&=\frac{1}{x_{k+1}-x_{k}}\bigg|
\left(
\int_{\alpha}^{x_{k}}p_{1}''(t)(x_{k}-t)dt - \int_{\alpha}^{x_{k}}p_{2}''(t)(x_{k}-t)dt
\right)
\bigg|
\notag \\
&\leq
\frac{x_{k}-\alpha}{x_{k+1}-x_{k}}\left(
\int_{\alpha}^{x_{k}}|p_{1}''(t)|dt + \int_{\alpha}^{x_{k}}|p_{2}''(t)|dt
\right)
\notag \\
& \leq 2\frac{x_{k}-\alpha}{x_{k+1} - x_{k}} h^{1/2} \|p''\|_{0,I^{-}\cup I^{+}} \notag \\
&\leq 2 h^{1/2} \|p''\|_{0,I^{-}\cup I^{+}}.
\end{align}

Having decomposed $\tilde J_1$ as the difference of a small term and a large term plus a finite term, we do the same for $\tilde J_2$.
First, with $\psi=m_2(x-x_{k+1})$ we have
\begin{align}\notag
&\left( \pi_h(p'_{2}- p'_{1})(x)\psi\right)' \\
\notag
&=h_k^{-1}
(\Delta'_{k+1}- \Delta'_{k})\psi
+h_k^{-1}\big[(\Delta'_{k+1}(x-x_k)-
\Delta'_{k}(x-x_{k+1})
\big]\psi'
\\& =m_2h_k^{-1}\big[
(\Delta'_{k+1}-\Delta'_k)(x-x_{k+1}) +\Delta'_{k+1}(x-x_{k})-\Delta'_{k}(x-x_{k+1})
\big]\notag\\
&=m_2h_k^{-1}\left(\sum_{i=1}^3I_i+\Delta'(\alpha)(x_{k+1}-x_{k})\right),
\end{align}
where
\begin{align*}
I_1&=(\Delta'(x_{k+1})-\Delta'(x_{k}))(x-x_{k+1}),\\
I_2&=(\Delta'(x_{k+1})-\Delta'(x_{k}))(x-x_k),\\
I_3&=(\Delta'(x_k)-\Delta'(\alpha))(x_{k+1}-x_k),
\end{align*}
and hence
\begin{align}\label{J21}
&\tilde J_{2}:=\frac{(x_{k}-\alpha)(p_{2}'(\alpha) - p_{1}'(\alpha))}{x_{k+1}-x_{k}}+\left( \pi_h(p'_{2}- p'_{1})(x)\psi(x)\right)'\\
&=h_k^{-1}((x_{k}-\alpha)[p']_\alpha+m_2h_k[p']_\alpha)+m_2h_k^{-1}\left(\sum_{i=1}^3I_i\right),
\end{align}
where the last term can be estimated as before.
Thus,
\begin{align}\label{J2T}
\tilde J_{2}= \tilde J_4+h_k^{-1}\left((x_{k}-\alpha)[p']_\alpha+m_2h_k[p']_\alpha\right)
\end{align}
with due to \eqref{uniform}
\begin{align*}
|\tilde J_4|=|m_2h_k^{-1}\left(\sum_{i=1}^3I_i\right)|\leq Ch^{1/2}\|p''\|_{0,I^{-}\cup I^{+}}.
\end{align*}

Now
\[
   \left(\bar{I}_{h}p - I_{h}p\right)'=\left(\bar{I}_{h}p - I^c_{h}p\right)'-\delta\psi'\]
   will be estimated as follows.
From \eqref{deferrLem2}, \eqref{J1T}, and \eqref{jmpGood}  we have
   \begin{align}\label{item1}
   \left(\bar{I}_{h}p - I_{h}p\right)'&=\left(\bar{I}_{h}p - I^c_{h}p\right)'-\delta\psi'\\
   \left(\bar{I}_{h}p - I^c_{h}p\right)'&=\tilde J_3-\tilde J_4-h_k^{-1}(x_{k}-\alpha)[p']_\alpha-h_k^{-1}[p]_\alpha-m_2[p']_\alpha-\frac{h_k^{-1}[p]_\alpha}{m_1}m_2\\
\notag   &=\tilde J_3-\tilde J_4-\tilde J_5
    \end{align}
    where
    \begin{equation}\label{J5}
    \tilde J_5:=\left(h_k^{-1}(x_k-\alpha)+h_k^{-1}\gamma+m_2+m_2h_k^{-1}\gamma/m_1\right) [p']_\alpha.
    \end{equation}
Thus
\begin{align}\label{conclusion}
\left(\bar{I}_{h}p - I_{h}p\right)'&=\mathcal{O}(h^{1/2}),
\end{align}
since $\tilde J_5=0$ by the way we defined $m_2$.

Gathering all the above local estimates and integrating, we conclude that there exists a constant $C>0$ independent of $h$ and $\alpha$ such that
\[
|\bar{I}_{h}p-I_{h}p|_{1,I^{+}}  \leq Ch\|p\|_{2,I^{-}\cup I^{+}}\quad \forall p\in {H}_\alpha^{2}(I).
\]
\end{proof}

Using Lemmas \ref{lemma3trm} and \ref{lemma3right1}, we obtain
\begin{theorem}\label{thm3.1}
There exists a constant $C>0$ independent of $h$ such that
\begin{equation}\label{interpolation err}
|p- I_{h}p|_{1,I}  \leq Ch\|p\|_{2,I^{-}\cup I^{+}}\quad \forall p\in {H}_\alpha^{2}(I) \text{ satisfying } \eqref{jmpGood}.
\end{equation}
\end{theorem}
Since our enriched finite element method is conforming, the convergence analysis is routine except for the step of checking the constant in the estimate to
be independent of maximum meshsize $h$ and the interface position $\alpha$.
\begin{theorem}\label{thm3.2}
Let $p$ be the exact solution  and $p_{h}$ be the approximate solution of (\ref{ExactWeak})  and (\ref{FEMweak}), respectively. Then
there exists a constant $C>0$ such that
\begin{equation}\label{main}
\|p-p_{h}\|_{0,I^-\cup I^+} + h\|p-p_{h}\|_{1,I^-\cup I^+} \leq Ch^{2}\|p\|_{2,I^{-}\cup I^{+}}.
\end{equation}
The constant $C$ does not depend on independent of $h$ and $\alpha$ but depends on the ratio  $\rho:=\frac{\beta^{*}}{\beta_{*}}$ with $\beta^{*} = \sup_{x\in [a,b]}\beta(x)$ and
$\beta_{*} = \inf_{x\in [a,b]}\beta(x)$.
\end{theorem}
\begin{proof}
Subtracting (\ref{ExactWeak}) from (\ref{FEMweak}), we have
\[
a(p-p_{h},q_{h}) = 0\quad\forall q_{h}\in\overline{S}_{h}.
\]
Then using the boundedness and coercivity properties of the bilinear form $a(\cdot,\cdot)$, we get
\begin{eqnarray*}
\beta_{*}|p-p_{h}|^2_{1,I} &\leq & a(p-p_{h},p-p_{h})  = a(p-p_{h},p-q_{h}) \\
&\leq& \beta^{*}|p-p_{h}|_{1,I}  |p-q_{h}|_{1,I},
\end{eqnarray*}
where $\beta^{*} = \sup_{x\in [a,b]}\beta(x)$  and $\beta_{*} = \inf_{x\in [a,b]}\beta(x)$.  Thus, by Cea's lemma and Theorem \ref{thm3.1}
\begin{eqnarray*}
|p-p_{h}|_{1,I}   &\leq& \frac{\beta^{*}}{\beta_{*}}\inf |p-q_{h}|_{1,I}  \\
&\leq& \frac{\beta^{*}}{\beta_{*}} |p-I_{h}p_{h}|_{1,I} \\
&\leq& Ch\|p\|_{2,I^{-}\cup I^{+}}.
\end{eqnarray*}
Then the usual duality argument leads to
\[
\|p-p_{h}\|_{0,I}  \leq  Ch^{2}\|p\|_{2,I^{-}\cup I^{+}}.
\]
\end{proof}

We note that the jump ratios  $\rho:=\frac{\beta^{*}}{\beta_{*}}$ are of moderate size for the wall model in the next section.

\begin{theorem}\label{2order} {\bf Second order accuracy at nodes.}
Suppose that $\beta\in C^1(a,\alpha)\cap C^1(\alpha,b)$ and $0\leq w\in C[a,b]$. Let $p$ be the exact solution  and $p_{h}$ be the approximate solution of (\ref{ExactWeak})  and (\ref{FEMweak}), respectively.
Then there exists a constant $C>0$ such that
\begin{equation}\label{2nd}
  |p(\xi)-p_h(\xi)|\leq Ch^2||p||_{2,I^-\cup I^+},\quad \xi=x_i,1\leq i\leq n-1.
  \end{equation}
where $C$ depends on certain norms of the Green's function at $\xi$.
\end{theorem}
\begin{proof}
Let $G(x,\xi),\xi\ne \alpha$ be the Green's function satisfying
\[ a(G(\cdot,\xi),v)=<\delta(x-\xi),v>,\quad v\in H^1_{0,\alpha}(a,b)\]
whose existence is guaranteed by the Lax-Milgram theorem, since in 1D point evaluation is a bounded operator.
Then from \cite{ATTChou2021,stakgold2011green}, we know without loss of generality  that for $\xi< \alpha$, $g=G(\cdot,\xi)\in H^2(\Omega)$, for $\Omega=(a,\xi),(\xi,x_k),(x_k,\alpha),(\alpha,x_{k+1}),(x_{k+1},b)$. Similar regularity holds if $\xi$ lies elsewhere. Since $g$ satisfies \eqref{jmpGood} we can use
{\it the local estimates} in Section 3 and conclude that
that there exists $ I_h g\in \bar S_h$ such that

\begin{equation}\label{gr} |g- I_h g|_{1,\Omega}\le C h ||g^{\prime\prime}||_{0,\Omega}
\end{equation}
for all the $\Omega$'s listed above.
Now
\[   e(x_i)=a(g, e)=a(g- I_h g,e)\] implies that
\begin{align*}
 |e(x_i)|&\le Ch||g||_{2,*}\,h||p||_{2,I^-\cup I^+}\\
  &\leq Ch^2||g||_{2,*}||p||_{2,I^-\cup I^+}
  \end{align*}
   where $||g||^2_{2,*}:=\sum||g||^2_{2,\Omega},$ the summation being over all the $\Omega$'s listed above.
\end{proof}

\section{Numerical Examples}
In this section, we test our method using the multi-layer porous wall model for the drug-eluting stents \cite{Pontrelli} that has been studied using the immersed finite element methods \cite{Wang, Zhang1, Zhang2,  Zhang}. In this one-dimensional wall model of layers, a drug is injected or released at an interface and gradually diffuses rightward. The concentration is thus discontinuous across the injection interface and continuous in the other layers. At all interface points, a zero-flux condition is imposed.
We run tests on both enriched linear and quadratic finite element spaces.

\subsection{Enriched Linear Elements}
In this subsection we test the efficiency of our method on three problems. In Problem 1, we place only one interface point to model the layer where the drug is delivered.  In Problem 2, we place two interfaces to model the layers where the concentration has continuously spread. Finally, in Problem 3 we combine the previous two cases and place three interface points to simulate the full wall model. In all three problems, we confirm in Table \ref{fig:P1}-Table \ref{fig:P3} the optimal order convergence in the broken $H^1$ and the $L^2$ norms. In addition, the nodal errors are shown to be second order in all these tables as well. We are interested in the behavior
of the condition numbers of the associated stiffness matrices. Following the viewpoint of the SGFEM \cite{Babuska1, Babuska2, Deng}, we compare the condition numbers in Problem 1 and Problem 3 (discontinuous solutions) with those in Problem 2 (continuous solution) \cite{ATTChou2021} that are displayed in Table \ref{fig:P2}. We can see that the condition numbers in Table \ref{fig:P1} and Table \ref{fig:P3} are comparable in the order of magnitude with those in Table \ref{fig:P2} for the same mesh sizes.

{\bf Problem 1.  Discontinuous Solution.} Consider the two-point boundary value problem with one interface point $\alpha_0=1/9$
\begin{equation}\label{Pro1}
\frac{\partial}{\partial x}\left(-D\frac{\partial u}{\partial x} + 2\delta u \right) + \gamma u = f \quad \text{ in }(0,1)
\end{equation}
subject to the the no-flux Neumann condition at $x=0$ and the Dirichlet condition at $x=0$:
\[
D_{0}u'(0) = 0, \quad u(1) = \frac{1}{3}.
\]
Here the drug reaction coefficient $\gamma=0$, and the drug diffusivity $D$ and the characteristic convection parameter $\delta$ are piecewise continuous with respect to  $[0,1/9]$ and $[1/9,1]$:
\[D(x) =
\begin{cases}
D_{0} = 1 & x\in[0,1/9] \\
D_1 = \frac{18(n-1)}{10n} & x\in[1/9,1];
\end{cases}\\
\]
\[
\delta(x)=
\begin{cases}
\delta_{0} = 0 & x\in[0,1/9] \\
 \delta = 0.5(9nD_1-8.1(n-1))& x\in[1/9,1].
\end{cases}
\]
Furthermore, at the interface point $\alpha_{0}$, one of the jump conditions is implicit
\begin{equation}\label{disContJump}
\begin{cases}
[u]_{\alpha_{0}} = \lambda D_{0}u'(\alpha_{0}), \\
-D_{0}u'(\alpha_{0})= -D_{1}u'(\alpha_{0}^{+}) + 2\delta_{1}u(\alpha_{0}^{+})
\end{cases}
\end{equation}
where $\lambda = 1/81(n-1)D_{0}$.
The exact solution
\[u(x) =
\begin{cases}
u_{0} = x^{n-1}/30, &  x\in[0,1/9],\\
 u_1 = x^n/3,& x\in[1/9,1]. \\
\end{cases}
\]
We test the effectiveness of the method with $n=4$ and with the enrichment function in \eqref{eqn:lin1}. The run results are displayed in Table \ref{fig:P1}.

\begin{table}[h]
\centering
    \begin{tabular}{|c||c|c|c|c|}
    \hline
    Problem 1 &  $L_{2}$ error & $H^{1}$ error & nodal error & condition number  \\
    \hline\hline
    $h=1/8$ & 1.43943e-03 &6.59920e-02 & 4.85121e-03 &  0.137850e+06 \\
    \hline
    $h=1/16$ & 3.40683e-04 & 3.24574e-02  & 1.01654e-03 & 0.143171e+06\\
    \hline
    $h=1/32$& 8.39493e-05 & 1.61603e-02 &  2.46808e-04 &  0.271847e+06 \\
    \hline
 $h=1/64$  & 2.09052e-05 &8.07152e-03  & 6.12768e-04 & 0.346975e+06  \\
 \hline
  $h=1/128$   & 5.22499e-06 & 4.03471e-03  & 1.53121e-04  &  0.125160e+07  \\
  \hline
   $h=1/256$ & 1.30868e-06 & 2.01723-03 &   3.82653e-06 & 0.839319e+07 \\
   \hline
    $h=1/512$ & 3.26874e-07 & 1.00859e-03  & 9.56539e-07 & 0.835384e+08  \\
    \hline
    order & $\approx 2$ & $\approx 1$  &  $\approx 2$ &  -  \\
    \hline
    \end{tabular}
    \caption{$L^{2}$-, broken $H^{1}$-, nodal errors, and condition numbers with discontinuous jump condition}
     \label{fig:P1}
\end{table}



{\bf Problem 2. Continuous Solution.} Consider the two-point boundary value problem with two interface points $\alpha_1=1/3,\alpha_2= 2/3$
\begin{equation}\label{P2}
\frac{\partial}{\partial x}\left(-D\frac{\partial u}{\partial x} + 2\delta u \right) + \gamma u = f \quad x\in(0,1)
\end{equation}
with the boundary conditions
\[
D_{0}u'(0) = 0 \quad u(1) = 0.
\]
Here with $n=4$
\[D(x) =
\begin{cases}
D_1 = \frac{18(n-1)}{10n} & x\in[0,1/3] \\
D_2 = \frac{6nD_1 - 2\delta}{3(n+1)}  & x\in[1/3,2/3] \\
D_3 = \frac{8\delta_2 - 3(n+1)D_2}{3(n+5)}  & x\in[2/3,1];
\end{cases}
\]

\[\delta(x)=
\begin{cases}
 \delta = 0.5(9nD_1-8.1(n-1)) & x\in[0,1/3]\\
 \delta_2 = 0.5(3(n+1)D_2 - 3nD_1 + 2\delta)  & x\in[1/3,2/3] \\
\delta_3 = 0.25(3(n-1)D_3 - 3(n+1)D_2 + 4\delta_2) & x\in[2/3,1],
\end{cases}
\]
 and $\gamma = 10, 1, 0.1$ in respective subintervals.
 At the interface points $\alpha_{i}$ for $i=1,2$, the solution $u$ is continuous and
\begin{equation}
\begin{cases}\label{contJump}
[u]_{\alpha_{i}} = 0, \\
-D_{i}u'(\alpha_{i}^{-}) + 2\delta_{i}u(\alpha_{i}^{-}) = -D_{i+1}u'(\alpha_{i}^{+}) + 2\delta_{i+1}u(\alpha_{i}^{+}).
\end{cases}
\end{equation}
The exact solution is

\[u(x) =
\begin{cases}
 u_1 = x^n/3& x\in[0,1/3] \\
u_2 = x^{n+1} & x\in[1/3,2/3] \\
u_3 = 3(1-x)x^{n+1}& x\in[2/3,1].
\end{cases}
\]

The enrichment function $\psi$ is well-known \cite{Babuska1, Deng, ATTChou2021}:
\begin{equation}
\psi(x) =
\begin{cases}
0   &  x\in [0,x_{k}]\\
\dfrac{(x_{k+1}-\alpha)(x_{k}-x)}{x_{k+1}-x_{k}} &  x\in [x_{k},\alpha]  \label{eqn:lin}\\
\dfrac{(\alpha-x_{k})(x-x_{k+1})}{x_{k+1}-x_{k}} &  x\in [\alpha,x_{k+1}]\\
0   &  x\in [x_{k+1},1].
\end{cases}
\end{equation}
The run results are displayed in Table \ref{fig:P2}.

\begin{table}[h]
\centering
    \begin{tabular}{|c||c|c|c|c|}
    \hline
    Problem 2 &  $L_{2}$ error & $H^{1}$ error  & nodal error  & condition number \\
    \hline\hline
    $h=1/8$ & 8.58406e-03 &2.91716e-01 &  2.07071e-02 & 0.127626e+05\\
    \hline
    $h=1/16$ & 2.11391e-03 & 1.46341e-01 & 4.56597e-03 & 0.109720e+06  \\
    \hline
    $h=1/32$& 5.30238e-04 & 7.35572e-02 & 1.11087-03 & 0.304583e+06  \\
    \hline
 $h=1/64$  & 1.32359e-04 &3.67855e-02  & 2.76462e-04 & 0.175135e+07 \\
 \hline
  $h=1/128$   & 3.31638e-05 & 1.84188e-02 & 6.91011e-05  & 0.511390e+07 \\
  \hline
   $h=1/256$ & 8.29035e-06 & 9.21011e-03 & 1.72680e-05 & 0.277080e+08 \\
   \hline
    $h=1/512$ & 2.07405e-06 & 4.60678e-03 & 431719e-06 & 0.825348e+08 \\
    \hline
    order & $\approx 2$ & $\approx 1$ &  $\approx 2$  &  - \\
    \hline
    \end{tabular}
    \caption{$L^{2}$-, broken $H^{1}$-, nodal errors, and condition numbers  with homogeneous jump conditions}
     \label{fig:P2}
\end{table}


{\bf Problem 3. Implicit and Explicit Conditions Both Present.} In this problem, we combine the interfaces of the last two problems. The interface points are $\alpha_{0} = 1/9$,
$\alpha_{1} = 1/3$ and $\alpha_{2} = 2/3$.
The two-point boundary value problem is
\begin{equation}\label{P3}
\frac{\partial}{\partial x}\left(-D\frac{\partial u}{\partial x} + 2\delta u \right) + \gamma u = f \quad x\in(0,1)
\end{equation}
subject to the boundary conditions
\[
D_{0}u'(0) = 0 \qquad u(1) = 0.
\]
The coefficients are defined as follows:

\[
D(x) =
\begin{cases}
D_{0} = 1 & x\in [0,1/9] \\
D_1 = \frac{18(n-1)}{10n}, & x\in [1/9,1/3] \\
D_2 = \frac{6nD_1 - 2\delta}{3(n+1)}  & x\in [1/3,2/3] \\
D_3 = \frac{8\delta_2 - 3(n+1)D_2}{3(n+5)}  & x\in [2/3,1];
\end{cases}
\]

\[\delta(x)=
\begin{cases}
\delta_{0} = 0 & x\in[0,1/9] \\
 \delta = 0.5(9nD_1-8.1(n-1)) & x\in[1/9,1/3] \\
 \delta_2 = 0.5(3(n+1)D_2 - 3nD_1 + 2\delta)  & x\in[1/3,2/3] \\
\delta_3 = 0.25(3(n-1)D_3 - 3(n+1)D_2 + 4\delta_2) & x\in [2/3,1];
\end{cases}
\]
$n=4$ and $\gamma =0, 10, 1, 0.1$ in respective subintervals.
The exact solution is

\[u(x) =
\begin{cases}
u_{0} = x^{n-1}/30 & [0,1/9]\\
 u_1 = x^n/3& [1/9,1/3] \\
u_2 = x^{n+1} & [1/3,2/3] \\
u_3 = 3(1-x)x^{n+1}& [2/3,1]
\end{cases}
\]
and satisfies the jump condition (\ref{disContJump}) at $1/9$ and   (\ref{contJump}) at the interface points $1/3$ and $2/3$.
For the discontinuous interface point $1/9$ we use the enrichment function defined in \eqref{eqn:lin1} and for the continuous interface points $1/3$ and $2/3$ we use the enrichment function in (\ref{eqn:lin}). The run results are displayed in Table \ref{fig:P3}.

\begin{table}[h]
\centering
    \begin{tabular}{|c||c|c|c|c|}
    \hline
    Problem 3 &  $L_{2}$ error & $H^{1}$ error & nodal error & condition number  \\
    \hline\hline
    $h=1/8$ & 8.58383e-03 &2.91715e-01 & 2.07048e-02 &  0.516955e+06 \\
    \hline
    $h=1/16$ & 2.11387e-03 & 1.46341e-01  & 4.56552e-03 & 0.207890e+06\\
    \hline
    $h=1/32$& 5.30234e-04 & 7.35577e-02 & 1.11075e-03 & 0.422880e+06 \\
    \hline
 $h=1/64$  & 1.32358e-04 &3.67868e-02  & 2.76434e-04 & 0.175140e+07  \\
 \hline
  $h=1/128$   & 3.31640e-05 & 1.84202e-02  & 6.90932e-05  &  0.511405e+07  \\
  \hline
   $h=1/256$ & 8.29083e-06 & 9.21222e-03 &  1.72659e-05 & 0.277086e+08 \\
   \hline
    $h=1/512$ & 2.07413e-06 & 4.61030e-03  & 5.16955e-06 & 0.129948e+09  \\
    \hline
    order & $\approx 2$ & $\approx 1$  &  $\approx 2$ &  -  \\
    \hline
    \end{tabular}
    \caption{$L^{2}$-, broken $H^{1}$-, nodal errors, and condition numbers with continuous and discontinuous  jump conditions}
     \label{fig:P3}
\end{table}

\subsection{Enriched Quadratic Elements}
In Problems 4 and 6, We test our method on the conforming $\mathbb P_2$ elements enriched by the enrichment function in \eqref{eqn:lin1}. All the conclusions in
subsection 4.1 hold, including the statements of optimal order convergence and condition numbers. The purpose of this section is to see what to expect when going on to higher order elements. The theory will be developed in another paper.

{\bf Problem 4.  Discontinuous Solution.} The BVP setting is the same as in Problem 1, Eq. \eqref{P1}, of the previous section.

We test the effectiveness of the method with $n=4$ and with the enrichment function in \eqref{eqn:lin1}. The run results are displayed in Table \ref{fig:P4}.

\begin{table}[h]
\centering
    \begin{tabular}{|c||c|c|c|}
    \hline
    Problem 4 &  $L_{2}$ error & $H^{1}$ error &  condition number  \\
    \hline\hline
    $h=1/8$ & 5.26785e-05 &2.78963e-03  &  0.173135e+07 \\
    \hline
    $h=1/16$ & 6.50740e-06 & 6.78494e-04  & 0.137419e+07\\
    \hline
    $h=1/32$& 8.10800e-07 & 1.68376e-04  &  0.435007e+07 \\
    \hline
 $h=1/64$  & 1.01260e-07 & 4.20145e-05   & 0.632200e+07  \\
 \hline
  $h=1/128$   & 1.26572e-08 & 1.05011e-05    &  0.126631e+08  \\
  \hline
   $h=1/256$ & 1.58255e-09 & 2.62599-06 &    0.690704e+08 \\
   \hline
    $h=1/512$ & 1.97777e-10 & 6.56287e-07   & 0.803646e+09  \\
    \hline
    order & $\approx 3$ & $\approx 2$   &  -  \\
    \hline
    \end{tabular}
    \caption{$L^{2}$-, broken $H^{1}$-, nodal errors and condition numbers with discontinuous jump conditions}
     \label{fig:P4}
\end{table}



{\bf Problem 5. Continuous Solution.} The BVP setting is exactly the same as in Problem 2, Eq. \eqref{P2}, of the previous section.
The enrichment function $\psi$ is \eqref{eqn:lin}. The run results are displayed in Table \ref{fig:P5}.

\begin{table}[h]
\centering
    \begin{tabular}{|c||c|c|c|c|}
    \hline
    Problem 2 &  $L_{2}$ error & $H^{1}$ error    & condition number \\
    \hline\hline
    $h=1/8$ & 6.27646e-04 &3.33102e-02  & 0.103879e+07\\
    \hline
    $h=1/16$ & 8.13417e-05 & 8.48195e-03 & 0.628185e+07  \\
    \hline
    $h=1/32$& 1.02475e-05 & 2.12831e-03 & 0.168823e+08  \\
    \hline
 $h=1/64$  & 1.28511e-06 &5.33221e-04   & 0.100139e+09 \\
 \hline
  $h=1/128$   & 1.60820e-07 & 1.33419e-04   & 0.270291e+09 \\
  \hline
   $h=1/256$ & 2.01127e-08 & 3.33693e-05  & 0.159784e+10 \\
   \hline
    $h=1/512$ & 2.51467e-09 & 8.34410e-06  & 0.432363e+10 \\
    \hline
    order & $\approx 3$ & $\approx 2$ &    - \\
    \hline
    \end{tabular}
    \caption{$L^{2}$-, broken $H^{1}$-, nodal errors and condition numbers  with homogeneous jump conditions}
     \label{fig:P5}
\end{table}


{\bf Problem 6. Implicit and Explicit Conditions Both Present.} The BVP setting is exactly the same as in Problem 3, Eq. \eqref{P3}, of the previous section.
The run results are displayed in Table \ref{P6}.
\begin{table}[h]
\centering
    \begin{tabular}{|c||c|c|c|}
    \hline
    Problem 6 &  $L_{2}$ error & $H^{1}$ error  & condition number  \\
    \hline\hline
    $h=1/8$ & 6.27649e-04 &3.33100e-02  &  0.252824e+07 \\
    \hline
    $h=1/16$ & 8.13414e-05 & 8.48190e-03   & 0.628277e+07\\
    \hline
    $h=1/32$& 1.02475e-05 & 2.12830e-03  & 0.168830e+08 \\
    \hline
 $h=1/64$  & 1.28510e-06 & 5.33218e-04  & 0.100141e+09  \\
 \hline
  $h=1/128$   & 1.60819e-07 & 1.33418e-04   &  0.270294e+09  \\
  \hline
   $h=1/256$ & 2.01127e-08 & 3.33692e-05  & 0.159785e+10 \\
   \hline
    $h=1/512$ & 2.51466e-09 & 8.34407e-06   & 0.125011e+11  \\
    \hline
    order & $\approx 3$ & $\approx 2$    &  -  \\
    \hline
    \end{tabular}
    \caption{$L^{2}$-, broken $H^{1}$-, nodal errors and condition numbers with continuous and discontinuous  jump conditions}
     \label{P6}
\end{table}

\newpage
\section{Concluding Remarks.}
Based on optimal order analysis, we derived a family of enrichment functions for the conforming $\mathbb P_1$ finite element, and the resulting enriched method
can approximate discontinuous solutions in optimal order in the broken $H^1$ and $L^2$ norms. Encouraged by the preliminary numerical results for the quadratic element
in subsection 4.2, we hope to extend the same approach to all $\mathbb P_i, i\geq 2$.

 Extension of our approach to higher dimensions is highly desirable. The tools we used in our approach in one dimension include Taylor's expansion,
extension operators in Sobolev spaces, and balancing the lower order large terms resulting from extension operators with the higher-order terms in the multipliers with the
enrichment function. All these have their counterparts in higher dimensions. The new ingredient in higher dimensions will include some contamination from the added
geometric complexity near the interface. An analogous 1D parameter called interface deviation $\epsilon$ was introduced in \cite{ATTChou2021} to mimic the geometric complexity
 (The enrichment function breaks at $\alpha-\epsilon$ instead of the interface point $\alpha$). We wish to further investigate the effect of this parameter on our present method.
\bibliographystyle{siam}
\bibliography{CHOU_ATTAN}

\end{document}